\documentclass[copyright]{eptcs}
\providecommand{\event}{QAPL 2014} 
\usepackage{breakurl}             
\usepackage{macros}

\title{Model Checking CSL for Markov Population Models
\thanks{
This research has also been partially funded by the German Research Council (DFG) as part of the 
Cluster of Excellence on Multimodal Computing and Interaction at Saarland University and the 
Transregional Collaborative Research Center "Automatic Verification and Analysis of Complex 
Systems" (SFB/TR 14 AVACS), 
by the National Natural Science Foundation of
    China (NSFC) under grant No.  61361136002, 61350110518, and the
    Chinese Academy of Sciences Fellowship for International Young
    Scientists (Grant No.  2013Y1GB0006). 
}}
\author{David Spieler
\institute{Modeling and Simulation Group\\Computer Science Department\\Saarland University\\ Saarbr\"ucken, Germany}
\email{spieler@cs.uni-saarland.de}
\and
Ernst Moritz Hahn \qquad\qquad Lijun Zhang
\institute{State Key Laboratory of Computer Science\\Institute of Software\\Chinese Academy of Sciences\\
Beijing, China}
\email{\quad hahn@ios.ac.cn \quad\qquad zhanglj@ios.ac.cn}
}
\def\titlerunning{Model Checking CSL for Markov Population Models}
\def\authorrunning{D. Spieler, E. M. Hahn \& L. Zhang}
\begin{document}
\maketitle

\begin{abstract}
Markov population models (MPMs) are a widely used modelling formalism
  in the area of computational biology and related areas. The
  semantics of a MPM is an \emph{infinite-state} continuous-time
  Markov chain.  In this paper, we use the established
  \emph{continuous stochastic logic} (CSL) to express properties of
  Markov population models. This allows us to express important
  measures of biological systems, such as probabilistic reachability,
  survivability, oscillations, switching times between attractor
  regions, and various others.  Because of the infinite state space,
  available analysis techniques only apply to a very restricted subset
  of CSL properties. We present a full algorithm for model checking
  CSL for MPMs, and provide experimental evidence showing that our
  method is effective.
\end{abstract}

\section{Introduction}
\label{sec:introduction}

In the context of continuous-time Markov chains (CTMCs), properties of
interest can be specified using continuous stochastic logic
(CSL)~\cite{AzizSSB00,BaierHHK03}. CSL is a branching-time temporal
logic inspired by CTL~\cite{EmersonC82}. It allows to reason about
properties of states (state labels), like the number of certain
molecules given in this state, about what may happen in the next
state (next operator), what may happen within a certain time (bounded until), what
may finally happen (unbounded until) or about the long-run average
behavior of a model (steady state). Because the underlying semantics
of a Markov population model is given as a CTMC, we can also use CSL
to reason about properties of such models, when interpreting CSL
formulae on the CTMC semantics.

We consider the complete set of CSL formulae, including the
steady-state operator and in certain cases also the unbounded until
operator~\cite{BaierHHK03}. The resulting logic can express (nested)
probabilistic properties such as \emph{``the long-run probability is
  at least 0.4 that we reach $\Psi$-states along $\Phi$-states
  within time interval [6.5,8.5] with a probability larger than
  $0.98$''} via $\steady_{\geq 0.4} (\probabilistic_{>0.98}( \Phi
\until^{[6.5,8.5]} \Psi))$. Using CSL, we can express many measures
important for biological models, including oscillation~\cite{BMM09}.

Previous works~\cite{Grassmann91,MoorselS94,MunskyK06} have already
considered techniques for the transient analysis of infinite-state
CTMCs. These techniques are based on \emph{truncation}. This means
that only a finite relevant subset of the states of the infinite CTMC
is taken into account. The extent to which these models are explored
depends on the rates occurring there, as well as on the time bound of
the transient analysis. Using truncation, techniques for the analysis
of finite CTMCs~\cite{Steward94} can be used for the analysis of
properties of infinite-state models.

In a previous publication~\cite{HahnHWZ09a}, we have extended these
results such that we were able to do approximate model checking for a
subset of CSL. This subset excluded the steady-state operator as well
as the unbounded until operator. We have implemented these techniques
in the model checker~\infamy~\cite{HahnHWZ09}. On the other hand, we
recently developed means to find subsets of states of Markov
population model CTMCs which contain the relevant steady-state
probability mass~\cite{dayar-hermanns-spieler-wolf10}. For each of
these states, we also obtain lower and upper bounds of the
steady state probability. These techniques have been implemented in
the tool~\geobound~\cite{geobound}.

In this paper, we combine~\geobound~and~\infamy, such that we can also
handle the CSL steady-state operator. In addition, we introduce advanced
truncation techniques which allow us to explore the model in a more
advanced way, leading to a smaller number of states being necessary to
check properties. By also taking into account not only the time bound
of CSL properties, but also the atomic propositions, we can further
restrict the state space to be explored. In certain cases, this also
allows us to handle the unbounded until operator. Using a ternary
logic~\cite{KatoenKLW07,Klink10}, in contrast to previous
publications, we can compute safe lower and upper bounds for
probabilities. In turn, we can decide exactly whether a certain
formula holds, does not hold or whether this cannot be decided on the
finite truncation of the current model. Apart from this, we have also
made some technical improvements, applying for instance to the
portability and robustness of~\infamy. We show the applicability of
the approach on a number of biological models.

\paragraph{Organization of the paper} We give background on Markov
population models, CTMCs and CSL in Section~\ref{sec:preliminaries},
and also recall the established CSL model checking algorithm for finite
CTMCs. In Section~\ref{sec:truncation}, we give the main contribution
of the paper, CSL model checking for infinite
CTMCs. Section~\ref{sec:experimental} reports experimental
results. Section~\ref{sec:related} gives related work and
Section~\ref{sec:conclusion} concludes.

\section{Preliminaries}
\label{sec:preliminaries}

\subsection{Ternary Logic}
We consider a ternary logic~\cite{KatoenKLW07,Klink10} with values
$\bternary := \{\vtrue, \vfalse, \vunknown\}$. With the ordering
$\vfalse < \ \vunknown < \vtrue$, $\bternary$ forms a complete
lattice. We interpret $\vand$ as the meet (``and'' operator), and
$\cdot^c$ as the complement (``not'') operation, with the usual
definitions. Other operators like $\vor$ (``or'' operator) can be
derived. Then, $\vtrue$ and $\vfalse$ can be interpreted as values
definitely true or false respectively, and $\vunknown$ is interpreted
as an unknown value. We give an overview of the truth values in
Table~\ref{tab:truth-values}. Consider a formula over a number of
values some of which are $\vunknown$. If the value of this formula is
different from $\vunknown$, we know that when inserting $\vfalse$ or
$\vtrue$ instead of some of these values, the result would still be
the same. This way, in some cases we can obtain a truth value even
though we do not known the truth values of some formula parts.

\begin{table}
\centering
\begin{tabular}{|c|ccc|}
  \hline
  $\vor$ & $\vfalse$ & $\vunknown$ & $\vtrue$ \\ \hline
  $\vfalse$ & $\vfalse$ & $\vunknown$ & $\vtrue$\\ 
  $\vunknown$ & $\vunknown$ & $\vunknown$& $\vtrue$\\
  $\vtrue$ & $\vtrue$& $\vtrue$& $\vtrue$\\ \hline
\end{tabular}\ \ \ 
\begin{tabular}{|c|ccc|}
  \hline
  $\vand$ & $\vfalse$ & $\vunknown$ & $\vtrue$ \\ \hline
  $\vfalse$ & $\vfalse$ & $\vfalse$ & $\vfalse$\\ 
  $\vunknown$ & $\vfalse$ & $\vunknown$& $\vunknown$\\
  $\vtrue$ & $\vfalse$& $\vunknown$& $\vtrue$\\ \hline
\end{tabular}\ \ \ 
\begin{tabular}{|c|c|}
  \hline
  & $C$ \\ \hline
  $\vfalse$ & $\vtrue$ \\
  $\vunknown$ & $\vunknown$ \\
  $\vtrue$ & $\vfalse$ \\ \hline
\end{tabular}
\caption{\label{tab:truth-values}Truth values of ternary logic operations.}
\end{table}

\subsection{Markov Chains}
We define some basics of Markov chains and CSL. Let $\AP$ denote a set
of atomic propositions. For a countable set $S$, a distribution $\mu$
over $S$ is a function $\mu: S\to [0,1]$ with $\mu(S) = 1$, and we
define $\mu(A) := \sum_{s\in A}\mu(s)$ for $A\subseteq S$.

\begin{definition}
  A \emph{labelled continuous-time Markov chain (CTMC)} is a tuple
  $\ctmc$ where $\states$ is a countable set of states, $\init$ is an
  initial state, $\labels: (\states \times \AP) \to \bternary$ is a
  labelling function, and $\R:(\states \times \states)\to \nnegreal$
  is the \emph{rate matrix}.
\end{definition}
In the stochastic process which the CTMC represents, $\init$ is the
state we start in. The labelling function tells us for each state $s$
and atomic proposition $a$ whether $a$ holds in $s$, does not hold in
$s$, or whether this is not known or not specified. We say that the
rate matrix is \emph{rate-bounded} if the supremum $\sup_{s \in
  \states}\R(s,\states)$ is finite, otherwise, it is called
\emph{rate-unbounded}. If $\R(s,s')>0$, we say that there is a transition
from $s$ to $s'$. For $s\in \states$, let $\post(s) := \{s'\mid
\R(s,s')>0\}$ denote the set of successor states of $s$. A state $s$
is called \emph{absorbing} if $\post(s)=\emptyset$. A CTMC is \emph{finitely
branching} if for each state $s$ the set of successors $\post(s)$ is
finite. In this paper, we consider rate-unbounded, finitely-branching
CTMCs which do not explode~\cite{Anderson91,Dijk88}. Roughly
speaking, if the CTMC explodes, then there is a positive, non-zero
probability of infinitely many jumps in finite time. On the contrary,
the fact that a CTMC does not explode implies that in finite time with
probability one only a finite number of states can be reached.

Transition probabilities in CTMCs are exponentially distributed
over time. The probability that an arbitrary transition is
triggered within time $t$ is given by $1-e^{-\R(s,\states) t}$, where
$\R(s,A):=\sum_{s'\in A}\R(s,s')$ for $A\subseteq \states$. The
probability of taking a particular transition from $s$ to $s'$ within
time $t$ is $\frac{\R(s,s')}{\R(s,\states)} \left(1-e^{-\R(s,\states)
    t}\right)$. State $s'$ is reachable from $s$ if there exists a
sequence of states $s_1, \ldots, s_n$ with $n\ge 1$, $s_1=s$,
$s_n=s'$, and for each $i=1, \ldots, n-1$ it is $\R(s_i,s_{i+1})>0$.

For a set of states $A$, by $\C[A]$ we denote the CTMC in which states
of $A$ have been made absorbing. For $\ctmc$ and $A \subseteq \states$
we have $\C[A] := (\states, \init, \R', \labels)$ with $\R'(s,s') :=
\R(s,s')$ for $s \notin A$ and $\R'(s,s') := 0$ else, for all $s,s' \in
S$. Given a rate matrix $\R$, we define the corresponding
\emph{infinitesimal generator matrix} $\Q$ such that for $s$, $s'$ if $s \neq
s'$ then $\Q(s,s') := \R(s,s')$ and we let $\Q(s,s) := - \sum_{s'\in S,
  s'\neq s} \R(s,s')$.

\paragraph{Paths and probabilistic measure}
Let $\ctmc$ be a CTMC. A path is a sequence $\mpath = s_1 t_1
s_2 t_2 \ldots$ satisfying $\R(s_i,s_{i+1})>0$, and $t_i \in
\nnegreal$. Paths are either infinite or have a last state $s_n$ with
$\R(s_n,S) = 0$. For the path $\mpath$ and $i\in \nat$, let
$\mpath[i]=s_i$ denote the $i$-th state, and let
$\delta(\mpath,i)=t_i$ denote the time spent in $s_i$. For $t\in
\nnegreal$, let $\mpath@t$ denote $\mpath[i]$ such that $i$ is the
smallest index with $t \le \sum_{j=1}^i t_j$. For $\C$, let
$\mpaths^\C$ denote the set of all paths, and $\mpaths^\C(s)$ denote the
set of all paths starting from $s$. A probability measure $\prob_s^\C$
on measurable subsets of $\mpaths^\C(s)$ is uniquely defined. We omit
the superscript $\C$ if it is clear from context.

\subsection{Markov Population Models}
In this paper, we will consider \emph{Markov population models} (MPMs),
a sub-class of CTMC, where states encode the number of individuals of
certain population types represented by a vector over the natural
numbers. Transitions between those states are defined by a change
vector that characterizes the successor state and a rate determined by
a propensity function that is evaluated in the predecessor state.

Formally, a MPM with $d$ population types is a CTMC $\ctmc$ with
$\states \subseteq \Nd$. The rate matrix $R$ will be induced by
\emph{transition classes}.
\begin{definition}[Transition Class]
  \label{def:tc}
  A \emph{transition class} is a tuple $(\alpha,v)$ where $\alpha: \Nd
  \rightarrow \Rp$ is the \emph{propensity function} and $v \in
  \Zd\setminus\{0\}$ is the \emph{change vector}.
\end{definition}
For the propensity functions, we will restrict to multivariate
polynomials in $\Nd$. Given a set of transition classes
$\{\tau_1,\dots,\tau_m\}$, we define the entries of $\R$ via $\R(x,x+v) = \sum_{\{j~|~v_j = v\}} \alpha_j(x)$ for $\tau_j
= (\alpha_j,v_j)$ and $1\leq j\leq m$. In order to ensure $\R$ is well
defined, we demand that for all states $x\in \states$ and all
transition classes $(\alpha_j,v_j)$ of our model $\alpha_j(x) > 0$
only holds if $x+v_j \in \states$. For the labelling $L$ we demand
that each state is labelled by those boolean expressions over
population counts and constants that evaluate to true. We also refer
to the populations by their name. For example, if we have a MPM with
populations $A$ ($x_0$) and $B$ ($x_1$), then state $(3,4)$ would be
labelled by e.g. $A \leq c$ for all $c \geq 3$, $A + B \leq 7$ and
$A^2 + B^2 < 30$ to name just a few.

\subsection{CSL Model Checking}
\label{sec:csl-model-checking}

We consider the logic CSL~\cite{BaierHHK03} interpreted over a ternary
logic~\cite{KatoenKLW07,Klink10}. Let $I=[t,t']$ be an interval with
$t,t'\in \nnegreal \cup \{\infty\}$ with $t'=\infty \Rightarrow t = 0$
and $t \le t'$. Let $p\in [0,1]$ and $\compare\ \in\{\le ,<,>,\ge
\}$. The syntax of state formulae ($\Phi$) and path formulae ($\phi$)
is:
\begin{eqnarray*}
  \Phi &=& a \mid \neg\Phi \mid \Phi\wedge \Phi \mid
  \probabilistic_{\compare p}(\phi) \mid \steady_{\compare p}(\Phi), \\
  \phi &=& \nextop^I \Phi \mid \Phi \until^I \Phi.
\end{eqnarray*}
Let $\cslfms$ be the set of all CSL formulae. The truth value
$\val[\C]{\cdot}: ((\states \cup \mpaths) \times \cslfms) \to \bternary$
of formulae is defined inductively in
Table~\ref{tab:csl-semantics}. If the model under consideration is
clear from the context, we leave out the index $\C$.
\begin{table}
  $\val{s,a} := \labels(s,a)$,
  $\val{s,\Phi_1 \wedge \Phi_2} := \val{s,\Phi_1} \vand
  \val{s,\Phi_2}$, $\val{s,\neg \Phi} := \val{s,\Phi}^c$, \\
  $\val{s,\probabilistic_{\compare p}(\phi)} {:=}
  \begin{cases}
    \vtrue, & \prob_s \{ \mpath {\mid} \val{\mpath,\phi}
    {=} \vtrue \} {\compare} p \wedge \prob_s \{ \mpath {\mid}
    \val{\mpath,\phi} {\neq} \vfalse \} {\compare} p, \\
    \vfalse, & \prob_s \{ \mpath {\mid} \val{\mpath,\phi}
    {=} \vtrue \} {\not\compare} p \wedge \prob_s \{ \mpath {\mid}
    \val{\mpath,\phi} {\neq} \vfalse \} {\not\compare} p, \\
    \vunknown & \mathrm{~else},
  \end{cases}$ \\
  $\val{s,\steady_{\compare p}(\Phi)} {:=}
  \begin{cases}
    \vtrue, & S_L(\Phi) {\compare} p \wedge S_U(\Phi) {\compare} p, \\
    \vfalse, & S_L(\Phi) {\not\compare} p \wedge S_U(\Phi)
    {\not\compare} p, \\
    \vunknown & \mathrm{~else},
  \end{cases}$ \\
  where \\
  $S_L(\Phi) := \lim\limits_{t \to \infty}\prob_s \{ \mpath \mid
  \val{\mpath@t,\Phi} = \vtrue \}$, \\
  $S_U(\Phi) :=\lim\limits_{t \to \infty} \prob_s \{ \mpath \mid
  \val{\mpath@t,\Phi} \neq \vfalse \}$, \\
  $\val{\sigma, \nextop^I \Phi} := (\vtrue \mathrm{~iff~}
  \delta(\mpath,0) \in I, \vfalse \mathrm{~else}) \vand
  \val{\mpath[1],\Phi}$, \\
  $\val{\sigma, \Phi_1 \until^I \Phi_2} {:=}
  \begin{cases}
      \vtrue, & \exists t {\in} I . \val{\sigma@t,
        \Phi_2} {=} \vtrue \wedge \forall t' {<} t . \val{\sigma@t',
        \Phi_1} {=} \vtrue, \\
      \vfalse, & \forall t {\in} I . \val{\sigma@t,
        \Phi_2} {=} \vfalse \vee \exists t' {<} t . \val{\sigma@t',
        \Phi_1} {=} \vfalse, \\
      \vunknown & \mathrm{~else}.
    \end{cases}$
  \caption{\label{tab:csl-semantics}CSL semantics.}
\end{table}
The model $\C$ satisfies a formula $\Phi$ if the initial state $\init$
does. We specify the \emph{unbounded until} as $\Phi_1 \until \Phi_2
:= \Phi_1 \until^{[0,\infty)} \Phi_2$ and the \emph{eventually
  operator} as $\Diamond^I \Phi := \mathit{true} \until^I \Phi$ where $\mathit{true}=a\vee \neg a$. Let
$B\subseteq \states$ be a set of states. For better readability, we
use the name of the set $B$ as an atomic proposition in formulae to
characterize that the system is in a state contained in $B$.

Model checking CSL formulae without the operators $\nextop$ and
$\steady$ on a finite CTMC over a ternary logic has already been used
before to handle a different abstraction technique for
CTMCs~\cite{KatoenKLW07,Klink10}. Adding routines to handle $\nextop$
and $\steady$ is straightforward.

\section{Model Checking based on Truncations}
\label{sec:truncation}

Now we discuss how to model check CSL formulae on infinite CTMCs. To
this end, our goal is to operate on a finite truncation instead of the
original infinite CTMC. In a nutshell, starting from the initial
state, we explore the states of the infinite model until we assume
that we have explored enough of them to obtain a useful result in the
following steps. Then, we remove all transitions from the states at
the border and set all of their atomic propositions to $\vunknown$. In
previous works~\cite{HahnHWZ09a} we already discussed some variants of
such a model exploration. In Subsection~\ref{sec:uniformisation} we
give another such technique. It is more efficient than the previous methods, as
it can explore model states in a more target-oriented way. Thus, it needs to
explore less states than the previous methods. Afterwards, we discuss how to
build a finite CTMC truncation for a nested CSL formula. Finally, in
Subsection~\ref{sec:nested}, we explain how to obtain results for the
infinite-state CTMC only using the finite submodel.

\subsection{Truncation-based Reachability Analysis}
\label{sec:uniformisation}

Given a set of states $\win_0$, a CTMC $\ctmc$ and CSL formula of the form
$\Phi = \probabilistic_{\compare p} (\Phi_1 \until \Phi_2)$, we want
to compute a finite submodel $\trunc{\win}$ sufficient to decide $\Phi$
on all states of $\win_0$. We define \emph{finite truncations} of a CTMC. 

\begin{definition}
  Let $\ctmc$ be a CTMC. Let $\win \subseteq \states$ be a finite subset of
  $\states$, and let $\overline{\win} := \post(\win) \setminus \win$. The
  \emph{finite truncation} of $\C$ is the finite CTMC $\truncdef{\win}$
  where $\labels_\win(s,\cdot) := \labels(s,\cdot)$ if $s \in \win$, and
  $L_\win(s, \cdot)=\ \vunknown$ else. The rate matrix is defined by
  $\R_\win(s,s') := \R(s,s')$ if $s\in \win$, and $\R_\win(s, \cdot) := 0$ else.
\end{definition}

We build the truncation of the model iteratively, using the high-level
(transition class) description of the model. Starting from states $A$,
we explore the model until for all $s \in \win_0$ the probability to
reach states in $\overline{\win}$ is below an accuracy $\epsilon$, which
we may choose as a fixed value or due to the probability bound $p$.

Algorithm~\ref{alg:transienttrunc} describes how we can obtain a
sufficiently large state set $\win$. For $s \in \states$, $s' \in
\overline{\win}$ and $t \in \nnegreal \cup \{\infty\}$ we use
$\transient_\C(s,t,s')$ to denote the probability that at time $t \in
\nnegreal$, the CTMC $\C$ is in state $s'$, under the condition that
it was in $s$ initially. For $t = \infty$, we let $\transient$ denote
the limit for $t \to \infty$ if this value exists. If $s$ will reach absorbing states with probability one, this
is the case. Further, for a set of absorbing states $B$, we let
$\reach_\C(s,t,B) := \sum_{s' \in B} \transient_\C(s,t,s')$ denote the
probability to reach $B$ within time $t \in \nnegreal \cup \{\infty\}$. Given a
fixed $s$ and $t$, we can compute $\transient_\C(s,t,s')$ for all $s'$ at once
effectively, and given $B$ and $t$ we can compute $\reach(s,t,B)$ for all $s$ at
once effectively~\cite{BaierHHK03}.

The algorithm is started on a CTMC $\C$ and a set of states $\win_0$, for
which we want to decide the property. We also provide the time bound
$t$ as well as the accuracy $\epsilon$. With $\hat{\win}$ we denote a
set of states for which the exploration algorithm may stop immediately, as
further exploration is not needed to decide the given property. For
$\Phi$ above and $I=[0,a],a\in\nnegreal \cup \{\infty\}$, we could
specify $\hat{\win}$ as the states which fulfill $\Phi_2 \vee (\neg
\Phi_1 \wedge \neg \Phi_2)$. For all paths of the model, the truth
value is independent of the states after such a state.

\begin{algorithm}
  \DontPrintSemicolon
  $\win := \win_0$ \;
  $\overline{\win} := \post(\win) \setminus \win$ \;
  \While {$\max_{s\in \win_0} \reach_{\C[\overline{\win}]}(s, t,
    \overline{\win} \setminus \hat{\win}) \geq \epsilon$}{
    choose $s$ from $\argmax_{s \in \win_0}
    \reach_{\C[\overline{\win}]}(s,t,\overline{\win} \setminus \hat{\win})$ \;
    \While {$\reach_{\C[\overline{\win}]}(s,t,\overline{\win} \setminus
      \hat{\win}) \geq \epsilon$}{
      choose $A \subseteq (\overline{\win} \setminus \hat{\win})$
      such that $\transient_{\C[\overline{\win}]}(s,t,A) \geq \epsilon$ \;
      $\win := \win \cup A$ \;
      $\overline{\win} := \post(\win) \setminus \win$ \;
    }
  }
  \Return $\win \cup (\post(\win) \cap \hat{\win})$

  \caption{\label{alg:transienttrunc}$\prog{TransientTrunc}(\C,\win_0,\hat{\win},t,\epsilon)$.}
\end{algorithm}

\subsection{Truncation-based Steady-state Analysis}
\label{sec:steady}

In the following, we will develop a technique to retrieve a finite subset
of states $\win \subseteq \states$ that contains most of the total steady-state
probability mass, i.e. $\sum_{c \in \win} \pi(c) > 1-\epsilon$ for a given 
$\epsilon < 1$. The next step will be to derive lower and upper bounds on 
the state-wise steady-state probabilities inside that window $\win$.

\paragraph{Geometric Bounds}
For the presented methodology to be applicable, we have to restrict
our models to \emph{ergodic} MPMs, since for steady-state analysis
the equilibrium distribution has to exist uniquely. 
Ergodicity can be verified by the means of Lyapunov functions and the 
following theorem. In the following, by $X(t)$ we refer to the stochastic process underlying
the MPM, and by $E$ to the expectation of a random variable.
\begin{definition}[Lyapunov Function]
A \emph{Lyapunov function} is a function $g: \Nd \rightarrow \Rp$.
\end{definition}
A suitable Lyapunov function which also we used for our experimental results is the
squared Euclidean norm $||\cdot||_2^2$ defined by $||x||_2^2 = x_1^2 + x_2^2 + \dots + x_d^2$
where $d$ is the number of population types.
\begin{theorem}[Tweedie~\cite{tweedie}]
\label{theorem:tweedie}
Assuming that $X(t)$ is \emph{irreducible}, it is also \emph{ergodic} and
uniquely determined by its infinitesimal generator iff there exists a
Lyapunov function $g^*$, a finite subset $\win\subseteq \states$, and a constant
$\lambda > 0$ such that
\begin{enumerate}
\item $\frac{\mathrm{d}}{\mathrm{d}t} E[g^*(X(t))~|~X(t) = x] \leq -\lambda$ for all $x\in\Nd\setminus\win$,
\item $\frac{\mathrm{d}}{\mathrm{d}t} E[g^*(X(t))~|~X(t) = x] < \infty$ for all $x\in\win$,
\item the set $\{x\in\Nd~|~g^*(x) \leq l\}$ is finite for all $l<\infty$.
\end{enumerate}
\end{theorem}
Therefore, given a MPM $\ctmc$ induced by a set of transition classes
$\{\tau_1,\dots,\tau_m\}$, we can use Theorem \ref{theorem:tweedie} for
a semi-decision procedure to check ergodicity by choosing candidates for Lyapunov
functions. Our experience has shown that in most cases rather simple functions, i.e.
multivariate polynomials of degree two, already suffice for usual models from systems biology
and queuing theory. Consequently, we restrict the choice of Lyapunov functions
to that class.

We can exploit Theorem~\ref{theorem:tweedie} also to retrieve the
aforementioned window $\win$ that encloses most of the steady-state
probability mass.
At first, let the \emph{drift} $d^*(x)$ in state $x\in\Nd$ be defined as
\[
d^*(x) := \frac{\mathrm{d}}{\mathrm{d}t} E[g^*(X(t))~|~X(t)=x] = (\Q g^*)(x).
\]
Due to the transition class induced structure of $\Q$, the drift is given as
\[
d^*(x) = \sum_{j=1}^{m}\alpha_j(x)(g^*(x+v_j) - g^*(x))
\]
and can easily be represented symbolically in a state $x$.
Since the propensity functions of our models as well as the Lyapunov
function are multivariate polynomials, so is the drift. The next step is
to retrieve a positive real number $c \geq \max_{x\in\Nd} d^*(x)$. In order to
find that global maximal drift, common global optimization techniques like
gradient based methods and simulated annealing can not be used, since there
is no guarantee to get the real global maximum. What we propose instead is
to solve $\nabla d^*(x)|_p = 0$ for all $p \in 2^{\{x_1,\dots,x_m\}}$, where $f(x)|_p$
denotes the projection of the multivariate function $f$ onto the subspace spanned
by $x_i \in p$ to retrieve all $K$ possible candidates $m_k$. 
In order to solve the equation systems $\nabla d^*(x)|_p = 0$, we suggest the use
of the polyhedral homotopy continuation method, which is guaranteed to find
all roots. For implementations, we refer to \cite{phom} and \cite{hom4ps}.
Finally, after restricting the candidate set to $M = \{m_k~|~m_k \in
\Rp^d\}$, we set $c = \max_{m \in M} d^*(m)$.
Please note, that due to the existence of a maximal
value $c$, the chosen Lyapunov function serves as a witness for ergodicity
assuming an irreducible MPM.
By scaling the Lyapunov function by $\frac{1}{c+\gamma}$
we retrieve the normalized drift
\[
d(x) = \frac{\mathrm{d}}{\mathrm{d}t} E [g(X(t))~|~X(t) = x]
\]
and using conditions 1 and 2 from Theorem \ref{theorem:tweedie} we get
\begin{equation}
\label{eq:drift}
d(x) \leq \frac{c}{c+\gamma} - \bar\chi_\win(x),
\end{equation}
where $\bar\chi_\win(x) = 1$ if $c\not\in\win$ and $0$ else. Multiplying
Inequality \ref{eq:drift} with $\pi(x)$ and summing over $x$ leads us
to
\[
\pi(\bar\win) = \sum_{x\not\in\win}\pi(x) \leq \frac{c}{c+\gamma}.
\]
Convergence of this sum is guaranteed for (infinite) ergodic MPMs as stated in \cite{glynn}.
Consequently, we can exploit this inequality by directly choosing $d(x)=\frac{\epsilon}{c}d^*(x)$,
i.e. setting $\lambda > 0$ such that $\epsilon=\frac{c}{c+\gamma}$,
to get $\pi(\win) > 1-\epsilon$ for
\[
\win = \{x\in\Nd~|~d(x) > \epsilon -1\}.
\]

\paragraph{State-wise Bounds}
Given our state space window $\win$, our next goal is to get lower ($l(x)$)
and upper ($u(x)$) bounds on the steady-state probabilities inside $\win$, i.e.
probability vectors $l$ and $u$ such that $l(x) \leq \pi(x) \leq u(x)$ for all
$c\in\win$. For this, we will employ the methodology developed by Courtois and
Semal \cite{courtoissemal}~\cite{courtois} which we have extended to
infinite state MPMs in \cite{dayar-hermanns-spieler-wolf10}:
\begin{theorem}[\cite{dayar-hermanns-spieler-wolf10}]
\label{theorem:sbounds}
Let $\Q$ be the infinitesimal generator of an ergodic CTMC $X(t)$ with countably infinite
state space $\states$ and let $\win\subseteq \states$ be a finite subset of the state space.
Further, we let the matrix $\mathbf{C}$ be the finite submatrix of $\Q$ containing exactly the states
in $\win$. If by $\mathbf{U}$ we refer to the uniformized CTMC of
$\mathbf{C}$, i.e.
\[
\mathbf{U} = \mathbf{I} + \alpha^{-1}\mathbf{C} \text{ with } \alpha >
\max_i -\mathbf{C}(i,i),
\]
then for all $x\in\win$ we have
\[
\min_j \pi^{\mathbf{U_j}}(x) \leq \frac{\pi(x)}{\sum_{x\in\win}
  \pi(x)} \leq \max_j \pi^{\mathbf{U_j}}(x)
\]
where $\pi^\mathbf{U_j}$ is the steady-state distribution of matrix
$\mathbf{U_j}$ which is matrix $\mathbf{U}$ made stochastic by
increasing column $j$.
\end{theorem}
As presented in the previous paragraph, we have $1-\epsilon < \sum_{x\in\win}\pi(x) \leq 1$.
Consequently, we can use the geometric bounding technique to
obtain the unconditional state-wise bounds from Theorem
\ref{theorem:sbounds}, that is to retrieve for all states $x\in\win$
\[
l(x) = (1-\epsilon) \min_j \pi^{\mathbf{U_j}}(x) \leq \pi(x) \leq
\max_j \pi^{\mathbf{U_j}}(x) = u(x).
\]
\paragraph{Geobound}
All these presented techniques, i.e., the retrieval of geometric bound via 
Lyapunov functions and the computation of state-wise steady-state bounds
via the methodology of the previous paragraph have been implemented in a
prototypical tool called \geobound~\cite{geobound}.

\subsection{Truncation-based CSL Model Checking}
\label{sec:nested}

Given a CTMC $\C$, we want to check whether $\C$ satisfies
$\Phi$. This is done in two phases. At first, we construct a finite
truncation that is sufficient to check the formula. To this end, we
employ an algorithm to determine suitable truncations. The states
explored depend on the specific CSL formula analyzed. The computation
works by recursive descent into sub-formulae. The most intricate
formulae are the probabilistic operators, for which we use the
technique from Section~\ref{sec:uniformisation}. After the
exploration, we can compute $\val{\init, \Phi}$ on the finite
truncation.

\begin{algorithm}
  \DontPrintSemicolon
  \SetKwBlock{Function}{function}{end}
  \Function({$\prog{Trunc}(\C,\win,\Phi)$}){
    \Switch{$\Phi$}{
      \lCase{$a$}{\Return $\win$} \;
      \lCase{$\neg\Psi$}{\Return $\prog{Trunc}{(\C,\win,\Psi)}$} \;
      \lCase{$\Phi_1\wedge\Phi_2$}{\Return
        $\prog{Trunc}(\C,\win,\Phi_1)\cup \prog{Trunc}(\C,\win,\Phi_2)$} \;
      \lCase{$\probabilistic_{\compare p}(\nextop^I\Psi)$}{\Return
        $\prog{Trunc}(\C,\win\cup\overline{\win},\Psi)$} \;
      \uCase{$\probabilistic_{\compare p}(\Phi_1\until^{[t,t']}\Phi_2)$}{
        $\win_t =  \prog{TransientTrunc}(\C,\win,\canstop(\Phi),t,\epsilon)$\\
        $\win_{t'} =  \prog{TransientTrunc}(\C,\win_t,\canstop(\Phi),t'-t,\epsilon)$\\
        \Return $\prog{Trunc}(\C,\win_{t'},\Phi_1) \cup
        \prog{Trunc}(\C,\win_{t'},\Phi_2)$}
      \lCase{$\steady_{\compare p}(\Psi)$}{ \Return
        $\prog{Trunc}(\C,\post(\win) \setminus \win,\Psi)$} \;
    }
  }
  \Return $\TC{\prog{Trunc}(\C, \{\init\}, \Phi)}$ \;
  
  \caption{\label{alg:explore-csl}$\prog{Truncate}(\ctmc,\Phi)$. CSL state space
    exploration.}
\end{algorithm}

Algorithm~\ref{alg:explore-csl} describes the exploration
component. Given a CTMC $\ctmc$ be a CTMC and a state formula $\Phi$,
we call $\prog{Trunc}(\C, \{\init\}, \Phi)$. Afterwards, we can use
the CSL model checking algorithm for a ternary logic on the model
obtained this way. With $\canstop(\Phi)$ we denote a set of states for
which we can stop the exploration immediately, as exemplified in
Section~\ref{sec:uniformisation}. For nested formulae, this value is
computed by a simple precomputation in a recursive manner.

We employ ternary CSL model checking on the finite model
obtained. However, we have already obtained state-wise bounds on the steady-state
probabilities beforehand using the approach presented in 
Section~\ref{sec:steady} and implemented in the tool \geobound. 
Thus, to obtain the lower bound
probabilities of $\steady_{\compare p}(\Psi)$, we sum up the
lower bound steady-state probabilities of states $s\in \win$ with
$\val{x,\Psi} = \vtrue$. For the upper probability bound, we sum upper
steady-state probabilities of states $s$ with $\val{x,\Psi} \neq \vfalse$
and add the probability $\epsilon$ that limits the steady-state
probability outside $\win$. The probabilities computed are the
probabilities for all states, because the model is ergodic.

\paragraph{Correctness} Consider a truncation
$\trunc{\win}$ constructed for the input CTMC $\mathcal{C}$ and a state formula $\Phi$. If
we obtain the truth value $\val{s,\Phi}
\neq \vunknown$ in $\mathcal{C}_\win$, then this is also the truth value in $\C$: The
correctness is independent of the exploration algorithm, which
plays a role for performance and applicability of the approach. If too
many states are explored, we may run out of memory, whereas if too few are
explored, we are unable to decide the value in the original model.
As a result, the correctness of the algorithm for CSL without steady state
  follows by giving a simulation relation~\cite[Definition
  3.4.2]{Klink10} between $\C$ and $\TC{\win}$ and \cite[Theorem
  4.5.2]{Klink10}. The correctness of the steady-state extension
  follows as we give safe upper and lower bounds in exactly the same
  way as it is done in \cite[Theorem 4.5.2]{Klink10}.

\section{Experimental Results}
\label{sec:experimental}

Using several case studies, we assess the effectiveness of our
technique. For that, we have combined the tool~\geobound~\cite{geobound} to compute
bounds on steady-state probabilities for Markov population models with
the infinite-state model checker $\infamy$~\cite{HahnHWZ09}. This way,
we can effectively handle the combination of models and properties
described in this paper. To show the efficiency of the approach, we
applied our tool chain on a number of models from different areas. The
results were obtained on an Ubuntu 10.04 machine with an Intel
dual-core processor at 2.66~GHz equipped with 3~GB of RAM. The tools
used are available\footnote{
  \texttt{\url{http://alma.cs.uni-saarland.de/?page_id=74}}}. Instead of the
truth value for the formula under considerations, in the result tables
we give intervals of the probability measure of the outer
formula. 
Note that in case of a single (ergodic) strongly connected component, there is a unique 
steady state distribution and it suffices to state a single pair of lower and upper 
probability bounds for a CSL formula with outer steady state operator since the validity is the
same for each state.
We make use of a derived operator for conditional steady-state measures defined as 
\[
\val{s,\steady_{\compare p}(\Phi_1~|~\Phi_2)} {=}
\begin{cases}
    \vtrue, & S_L(\Phi_1{\mid}\Phi_2) {\compare} p {\wedge} S_U(\Phi_1{\mid}\Phi_2) {\compare} p, \\
    \vfalse, & S_L(\Phi_1{\mid}\Phi_2) {\not\compare} p {\wedge} S_U(\Phi_1{\mid}\Phi_2)
    {\not\compare} p, \\
    \vunknown & \mathrm{~else,}
  \end{cases}
\]
where
\begin{eqnarray*}
S_L(\Phi_1~|~\Phi_2) &:=& \frac{S_L(\Phi_1\land\Phi_2)}{S_U(\Phi_2)}, \\
S_U(\Phi_1~|~\Phi_2) &:=& \frac{S_U(\Phi_1\land\Phi_2)}{S_L(\Phi_2)}.
\end{eqnarray*}

\paragraph{Protein Synthesis  \cite{GossP98}}
We analyze the MPM encoding protein synthesis, as depicted in Table
\ref{tab:proteinsyn_tc}. In biological cells, each protein ($P$, $x_2$) is encoded by a
gene ($G$, $x_1$). If the gene is active ($G=1$), the corresponding protein will be
synthesized with rate $\nu=1.0$. Proteins may degenerate with rate $\delta=0.02$ and thus 
disappear after a time. The gene switches from active state to inactive ($G=0$) with rate $\mu=5.0$
and vice versa with rate $\lambda=1.0$. Note that in a previous paper, this
model has been presented as a stochastic Petri net (SPN). Often, transition class models 
and SPN (without zero-arcs) can trivially be encoded within each other.
\begin{table}[htdp]
\begin{center}
\begin{tabular}{|p{0.5cm}|p{2cm}|p{2cm}|p{2cm}|}
\hline
\multicolumn{1}{|c|}{$j$\tchead} & \multicolumn{1}{c|}{$\tau_j$} & \multicolumn{1}{c|}{$\alpha_j(x)$} & \multicolumn{1}{c|}{$v^{(j)}$} \\
\hline
\rtab \tabhsp $1$ & \rtab$(\alpha_1,v_1)$ & \rtab$\lambda (1-x_1)$ & \rtab$e_1^T$  \\
\rtab $2$ & \rtab$(\alpha_2,v_2)$ & \rtab$\mu x_1$         & \rtab$-e_1^T$  \\
\rtab $3$ & \rtab$(\alpha_3,v_3)$ & \rtab$\nu x_1$         & \rtab$e_2^T$     \\
\rtab $4$ & \rtab$(\alpha_4,v_4)$ & \rtab$\delta x_2$      & \rtab$-e_2^T$  \\
\hline
\end{tabular}
\end{center}
\caption{\label{tab:proteinsyn_tc}Transition classes of the protein synthesis model.}
\end{table}

We consider the property that on the long run, given that 
there are more than 20 proteins, a state with 20 or less proteins is most likely 
(with a probability of at least $0.9$) reached within $t$ time units:
\[
\steady_{>p} (\probabilistic_{>0.9}(\Diamond^{[0,t]} P \leq 20) ~|~ P
> 20).
\]

\begin{table*}[tbh]
  \caption{\label{tab:protein-synthesis-2-results}Protein synthesis results.}
  \begin{center}
    \small{
      \begin{tabular}{|p{1cm}|p{1cm}|p{1cm}|p{2cm}|p{2cm}|p{1cm}|p{3cm}|}
        \hline
        \multicolumn{1}{|c|}{\multirow{2}{*}{$\epsilon$}} & \multicolumn{1}{c|}{\multirow{2}{*}{$t$}} & \multicolumn{1}{c|}{\multirow{2}{*}{depth}} & \multicolumn{2}{c}{time}  & \multicolumn{1}{|c|}{\multirow{2}{*}{$n$}}   & \rtab probability bounds \\ \cline{4-5}
        &     &       & \ctab \geobound & \ctab \infamy          &       & \rtab [$S_L$, $S_U$] \\ \hline
        
        \rtab\multirow{3}{*}{$10^{-1}$} & \rtab 10 & \rtab 8 & \rtab \multirow{3}{*}{0.9}      & \rtab 0.1  & \rtab 217   & \rtab [0.002, 1.0] \\
                   & \rtab 20  & \rtab 8     &          & \rtab 0.1             & \rtab 217   & \rtab [0.003, 1.0] \\
                   & \rtab 60  & \rtab 8     &          & \rtab 0.1             & \rtab 217   &  \rtab[0.004, 1.0] \\ \hline
                   
        \rtab\multirow{3}{*}{$10^{-3}$}   & \rtab 10 & \rtab 5     & \rtab \multirow{3}{*}{3.2} & \rtab 0.3             & \rtab 1531 & \rtab [0.144, 1.0] \\
                   & \rtab 20  & \rtab 5     &          & \rtab 0.4             & \rtab 1531  & \rtab [0.259, 0.816] \\
                   & \rtab 60  & \rtab 5     &          & \rtab 0.7             & \rtab 1531  & \rtab [0.317, 1.0] \\ \hline
                   
        \rtab\multirow{3}{*}{$10^{-6}$}   & \rtab 10  & \rtab 3 & \rtab \multirow{3}{*}{7.8} & \rtab 34.3            & \rtab 46431 & \rtab [0.451350, 0.454779] \\
                   & \rtab 20  & \rtab 3     &          & \rtab 67.0            & \rtab 46431 & \rtab [0.813116, 0.817401] \\
                   & \rtab 60  & \rtab 3     &          & \rtab 257.5           & \rtab 46431 & \rtab [0.997642, 1.0] \\ \hline 
      \end{tabular}
    }
  \end{center}
\end{table*}
We give results in Table~\ref{tab:protein-synthesis-2-results}. The
shortcut $\canstop$ of Algorithm~\ref{alg:explore-csl} was not used
for the analysis. With ``depth'' we specify the number of states of
the shortest path from the initial state to any other state of the
truncation. The runtimes of \geobound\ and \infamy\ is given in
seconds. The rate of decay of proteins depends on the number of
proteins existing. For the states on the border of $\win$, we have
large rates back to existing states. Because of this, for the given
parameters the state space exploration algorithm does not need to
explore further states, and the total number of states explored $n$
does not increase with the time bound. To obtain different $n$, we
would have needed to choose extremely large time bounds, for which
analysis would be on the one hand infeasible and on the other hand
would lead to results extremely close to $1$. The lower and upper
bounds are further away than $\epsilon$. This results, because we have
to divide by $S_U(\Phi_2)$ for the lower and by $S_L(\Phi_2)$ for the
lower bound. In turn, this may lead to a much larger error than
$\epsilon$.

\paragraph{Gene Expression~\cite{Thattai}}
Next, we analyze a network of chemical reactions where a gene is transcribed into 
mRNA ($M$) with rate $\lambda=25.0$ and the mRNA is translated into proteins ($P$) with
rate $\mu=1.0$. Both populations can degrade at rates $\delta_M=2.0$ and $\delta_P=1.0$, respectively. 
The corresponding transition classes are listed in Table~\ref{tab:geneexp_tc}. 
\begin{table}[htdp]
\begin{center}
\renewcommand{\arraystretch}{1.2}
\begin{tabular}{|p{0.5cm}|p{2cm}|p{2cm}|p{2cm}|}
\hline
\multicolumn{1}{|c|}{$j$\tchead} & \multicolumn{1}{c|}{$\tau_j$} & \multicolumn{1}{c|}{$\alpha_j(x)$} & \multicolumn{1}{c|}{$v^{(j)}$} \\
\hline
\rtab$1$\tabhsp & \rtab$(\alpha_1,v_1)$ & \rtab$\lambda$      & \rtab$e_1^T$     \\
\rtab$2$ & \rtab$(\alpha_2,v_2)$ & \rtab$\mu x_1$      & \rtab$e_2^T$    \\
\rtab$3$ & \rtab$(\alpha_3,v_3)$ & \rtab$\delta_M x_1$ & \rtab$-e_1^T$   \\
\rtab$4$ & \rtab$(\alpha_4,v_4)$ & \rtab$\delta_P x_2$ & \rtab$-e_2^T$    \\  
\hline
\end{tabular}
\end{center}
\caption{\label{tab:geneexp_tc}Transition classes of the gene expression model.}
\end{table}
The property of interest is the steady-state probability of
leaving a certain set of states $W$ enclosing more than 80\%
of the steady-state probability mass most likely within $t$ time units,
i.e.,
\[
\steady_{>p} (\probabilistic_{>0.9}(\Diamond^{[0,t]} \lnot W) ~|~ W),
\]
where
\[
W := M>5 \land M<20 \land P>5 \land P<20
\]
with $\steady_{>0.8} (W)=\vtrue$.

\begin{table*}[tbh]
  \caption{\label{tab:transcription-2-results}Gene expression results.}
  \begin{center}
    \small{
      \begin{tabular}{|p{1.3cm}|p{1cm}|p{1cm}|p{2cm}|p{2cm}|p{1cm}|p{3cm}|}
        \hline
        \multicolumn{1}{|c|}{\multirow{2}{*}{$\epsilon$}} & \multicolumn{1}{c|}{\multirow{2}{*}{$t$}} & \multicolumn{1}{c|}{\multirow{2}{*}{depth}} & \multicolumn{2}{c}{time}  & \multicolumn{1}{|c|}{\multirow{2}{*}{$n$}}   & \rtab probability bounds \\ \cline{4-5}
        &     &       & \ctab \geobound & \ctab \infamy          &       & \rtab [$S_L$, $S_U$] \\ \hline
        
        \rtab\multirow{3}{*}{$10^{-1}$} & \rtab 2   & \rtab 24     & \rtab \multirow{3}{*}{3.4} & \rtab 5.2 & \rtab 2558 & \rtab [0.01, 0.2] \\
                   & \rtab 4   & \rtab 24     &          & \rtab 6.0  & \rtab 2558   & \rtab [0.3, 0.6] \\
                   & \rtab 8   & \rtab 24     &          & \rtab 8.5  & \rtab 2558   & \rtab [0.8, 1.0] \\ \hline
                   
        \rtab\multirow{3}{*}{$5\cdot 10^{-2}$} & \rtab 2   & \rtab 20  & \rtab \multirow{3}{*}{6.1} & \rtab 11.9 & \rtab 3663  & \rtab [0.015, 0.078] \\
                   & \rtab 4   & \rtab 20     &          & \rtab 15.4 & \rtab 3663  & \rtab [0.34, 0.46] \\
                   & \rtab 8   & \rtab 20     &          & \rtab 22.1 & \rtab 3663  & \rtab [0.90, 1.0] \\ \hline
                   
        \rtab\multirow{3}{*}{$10^{-2}$}       & \rtab 2   & \rtab 15   & \rtab \multirow{3}{*}{8.5}      & \rtab 99.3 & \rtab 11736 & \rtab [0.015, 0.029] \\
                   & \rtab 4   & \rtab 15     &          & \rtab 139.9 & \rtab 11736 & \rtab [0.37, 0.40] \\
                   & \rtab 8   & \rtab 15     &          & \rtab 219.5  & \rtab 11736 & \rtab [0.97, 1.0] \\ \hline 
      \end{tabular}
    }
  \end{center}
\end{table*}

\begin{table*}[tbh]
  \caption{\label{tab:transcription-3-results}Gene expression - comparison of methods.}
  \begin{center}
    \small{
      \begin{tabular}{|p{1cm}|p{1.5cm}|p{1.5cm}|p{1.5cm}|p{1.5cm}|p{1.5cm}|p{1.5cm}|}
        \hline
        \ctab \multirow{2}{*}{$T$} & \multicolumn{2}{c|}{FSP} & \multicolumn{2}{c|}{Advanced}
        & \multicolumn{2}{c|}{Advanced+AP} \\ \cline{2-7}
        &    \ctab time & \ctab $n$ & \ctab time & \ctab $n$ & \ctab time & \ctab $n$ \\ \hline
        \rtab 1  & \rtab 0.3  & \rtab 1223  & \rtab 0.5  & \rtab 803   & \rtab 0.3  & \rtab 495 \\
        \rtab 2  & \rtab 0.8  & \rtab 1889  & \rtab 1.1  & \rtab 1257  & \rtab 0.6  & \rtab 838 \\
        \rtab 3  & \rtab 1.2  & \rtab 2209  & \rtab 1.5  & \rtab 1460  & \rtab 0.8  & \rtab 945 \\
        \rtab 4  & \rtab 1.5  & \rtab 2344  & \rtab 1.9  & \rtab 1557  & \rtab 0.9  & \rtab 971 \\
        \rtab 5  & \rtab 1.8  & \rtab 2483  & \rtab 2.1  & \rtab 1610  & \rtab 0.9  & \rtab 974 \\
        \rtab 6  & \rtab 2.1  & \rtab 2483  & \rtab 2.4  & \rtab 1647  & \rtab 1.0  & \rtab 974 \\
        \rtab 7  & \rtab 2.4  & \rtab 2554  & \rtab 2.7  & \rtab 1674  & \rtab 1.1  & \rtab 974 \\
        \rtab 8  & \rtab 2.7  & \rtab 2554  & \rtab 2.9  & \rtab 1690  & \rtab 1.2  & \rtab 974 \\
        \rtab 9  & \rtab 3.0  & \rtab 2626  & \rtab 3.2  & \rtab 1707  & \rtab 1.2  & \rtab 974 \\
        \rtab 10 & \rtab 3.2  & \rtab 2626  & \rtab 3.3  & \rtab 1720  & \rtab 1.3  & \rtab 974 \\
        \hline
      \end{tabular}
    }
  \end{center}
\end{table*}

The results are stated in
Table~\ref{tab:transcription-2-results}. Similar to the protein
synthesis case study, we see that there is no increase in the number
of states, because the window size already comprises enough states for
the transient analysis. In Table~\ref{tab:transcription-3-results} we
consider results for the subformula
$\probabilistic_{>0.9}(\Diamond^{[0,t]} \lnot W)$. We compare the
methods to explore states for the transient until described in this
paper (Advanced) with the finite state projection~\cite{MunskyK06}
(FSP) previously used in \infamy. We see that the time needed is
comparable, but the new algorithm needs to explore less states. This
is the case because with the method introduced here when building the
finite truncation we have more control into which direction we
explore. In contrast, the FSP explores the model into all directions
at the same time, until enough precision is reached. When we use the
shortcut $\canstop$ (Advanced+AP), we will not explore states further
in which $\lnot W$ holds. When exploring the model, for larger time
bounds there is some point at which there are almost only states of which all
successors have been completely explored and states for which $\lnot
W$ holds. Thus, the maximal number of states to be explored is
constant with this optimization, except for very large time bounds.

For the protein synthesis, there is almost no difference in the number
of states needed by the new method and FSP. Because it has only one
infinite variable, there is just one direction to be explored. Thus,
the new method performs worse, as it needs more effort to explore
into this direction.

Using the shortcut method also allows us to handle the formula
$\probabilistic_{>0.9}(\Diamond \lnot W)$, involving the unbounded
until operator. Using a precision of $10^{-6}$, we needed a total
time of 1.5 seconds, reached $974$ states and obtained a reachability
probability almost $1.0$. Using a precision of $10^{-10}$, we needed
3.7 seconds and explored $1823$ states. For the computation of
unbounded until probabilities, efficient specialized algorithms were
used, which explains that we needed less time than for some of the
time bounded experiments.

\paragraph{Exclusive Switch \cite{exswitch}}
The exclusive switch is a gene regulatory network with one promotor region 
shared by two genes. That promotor region can either be unbound ($G=1, G.P_1 = 0, G.P_2 = 0$)
or bound by a protein expressed by gene 1 ($G=0, G.P_1=1, G.P_2=0$)
or gene 2 ($G=0, G.P_1=0, G.P_2=1$). If the promotor is unbound,
both proteins are expressed, each with rate $\rho=0.05$, otherwise only
the protein that is currently bound to the promotor is produced
at rate $\rho$. Proteins degrade at rate $\delta=0.005$. Binding happens
at rate $\lambda=0.01$ and unbinding at rate $\mu=0.008$. The transition class
structure is given in Table \ref{tab:exswitch_tc}.

\begin{table}[htdp]
\begin{center}
\begin{tabular}{|p{0.5cm}|p{2cm}|p{2cm}|p{2cm}|} \hline
\multicolumn{1}{|c|}{$j$\tchead}  & \multicolumn{1}{c|}{$\tau_j$} & \multicolumn{1}{c|}{$\alpha_j(x)$} & \multicolumn{1}{c|}{$v^{(j)}$} \\ \hline

\rtab$1$\tabhsp & \rtab$(\alpha_1,v_1)$       & \rtab$\rho x_3$        & \rtab$e_1^T$  \\
\rtab$2$  & \rtab$(\alpha_2,v_2)$       & \rtab$\rho x_3$        & \rtab$e_2^T$ \\
\rtab$3$  & \rtab$(\alpha_3,v_3)$       & \rtab$\delta x_1$      & \rtab$-e_1^T$  \\
\rtab$4$  & \rtab$(\alpha_4,v_4)$       & \rtab$\delta x_2$      & \rtab$-e_2^T$  \\
\rtab$5$  & \rtab$(\alpha_5,v_5)$       & \rtab$\lambda x_1 x_3$ & \rtab$(-e_1 - e_3 + e_4)^T$\\
\rtab$6$  & \rtab$(\alpha_6,v_6)$       & \rtab$\lambda x_2 x_3$ & \rtab$(-e_2 - e_3 + e_5)^T$\\
\rtab$7$  & \rtab$(\alpha_7,v_7)$       & \rtab$\mu x_4$         & \rtab$(-e_4 + e_1 + e_3)^T$\\
\rtab$8$  & \rtab$(\alpha_8,v_8)$       & \rtab$\mu x_5$         & \rtab$(-e_5 + e_2 + e_3)^T$\\
\rtab$9$  & \rtab$(\alpha_9,v_9)$       & \rtab$\rho x_4$        & \rtab$e_1^T$  \\
\rtab$10$ & \rtab$(\alpha_{10},v_{10})$  & \rtab$\rho x_5$        & \rtab$e_2^T$ \\ \hline
\end{tabular}
\end{center}
\caption{\label{tab:exswitch_tc}Transition classes of the exclusive switch.}
\end{table}

This system has two attractor regions, i.e., two spatial maxima in the
steady-state probability distribution over the protein levels, one at 
$P_1=10, P_2=0$ and the other one at $P_1=0, P_2=10$. We are interested
in the switching time between these two regions. For this, we 
estimate the time needed for a 90\%-quantile of the steady-state probability mass
of one of the two attractors to reach the other attractor region. More 
precisely, let
\begin{eqnarray*}
\mathit{start} &:=& ||(P_1,P_2) - (10,0)||_2^2 \leq 4 \\
\mathit{end} &:=& ||(P_1,P_2) - (0,10)||_2^2 \leq 4.
\end{eqnarray*}
Then, the formula to check is
\[
\steady_{>p} (\probabilistic_{>0.9}(\Diamond^{[0,t]} \mathit{end})~|~ \mathit{start}) .
\]
Note that since the model is symmetric we only have to check one formula
from one attractor to the other.
The corresponding results are depicted in Table \ref{tab:exswitch-2-results}.
\begin{table*}[tbh]
  \caption{\label{tab:exswitch-2-results}Exclusive switch results.}
  \begin{center}
    \small{
      \begin{tabular}{|p{1.3cm}|p{1cm}|p{1cm}|p{2cm}|p{2cm}|p{1cm}|p{3cm}|}
        \hline
        \multicolumn{1}{|c|}{\multirow{2}{*}{$\epsilon$}} & \multicolumn{1}{c|}{\multirow{2}{*}{$t$}} & \multicolumn{1}{c|}{\multirow{2}{*}{depth}} & \multicolumn{2}{c}{time}  & \multicolumn{1}{|c|}{\multirow{2}{*}{$n$}}   & \rtab probability bounds \\ \cline{4-5}
        &     &       & \ctab \geobound & \ctab \infamy          &       & \rtab [$S_L$, $S_U$] \\ \hline
        
        \rtab\multirow{4}{*}{$10^{-1}$}  & \rtab 7700   & \rtab 14 & \rtab \multirow{4}{*}{5.8} & \rtab 47.9 & \rtab 3414 & \rtab [0.2, 0.7] \\
                   & \rtab 7800   & \rtab 14 &         & \rtab 48.0 & \rtab 3414  & \rtab [0.3, 0.9] \\
                   & \rtab 7900   & \rtab 14 &         & \rtab 47.6 & \rtab 3414  & \rtab [0.5, 1.0] \\ 
                   & \rtab 8000   & \rtab 14 &         & \rtab 48.1 & \rtab 3414  & \rtab [0.6, 1.0] \\ \hline
                   
        \rtab\multirow{4}{*}{$5\cdot 10^{-2}$} & \rtab 7700 & \rtab 12 & \rtab \multirow{4}{*}{6.9}      & \rtab 128.3 & \rtab 4848 & \rtab [0.26, 0.49] \\
                   & \rtab 7800   & \rtab 12 &         & \rtab 129.7       & \rtab 4848  & \rtab [0.43, 0.70] \\
                   & \rtab 7900   & \rtab 12 &         & \rtab 130.5       & \rtab 4848  & \rtab [0.64, 0.98] \\
                   & \rtab 8000   & \rtab 12 &         & \rtab 131.0       & \rtab 4848  & \rtab [0.83, 1.0] \\ \hline
                   
        \rtab\multirow{4}{*}{$10^{-2}$} & \rtab 7700    & \rtab 8 & \rtab \multirow{4}{*}{86.2} & \rtab 1881.6 & \rtab 14806 & \rtab [0.30, 0.35] \\
                   & \rtab 7800   & \rtab 8 &          & \rtab 1904.5      & \rtab 14806 & \rtab [0.50, 0.56] \\
                   & \rtab 7900   & \rtab 8 &          & \rtab 1930.1      & \rtab 14806 & \rtab [0.75, 0.82] \\
                   & \rtab 8000   & \rtab 8 &          & \rtab 1942.9      & \rtab 14806 & \rtab [0.96, 1.0] \\ \hline 
      \end{tabular}
    }
  \end{center}
\end{table*}

From these results we may conclude that in half of the cases, most likely the switching time 
between the attractor regions is at most 7800 time units, while in almost all cases the
switching time is most likely below 8000 time units, assuming the system has stabilized to a 
steady state.

\section{Related Work} 
\label{sec:related}
The techniques of this paper are derived from combinations of our
previous works~\cite{HahnHWZ09,HahnHWZ09a} and \cite{dayar-hermanns-spieler-wolf10,geobound}. 
This work has been inspired and is related by a number of other works.

Finite state projection (FSP) by Munsky and Khammash~\cite{MunskyK06}
is closely related. The method also works on building a finite
truncation of the original model. The proofs given work for general
truncations, but in their publications they always use an algorithm
which explores the model in a breadth-first way. They consider
time-bounded reachability and no logic like CSL. Adaptive
uniformization for CTMCs was introduced by van Moorsel and
Sanders~\cite{MoorselS94}. In this approach uniformization is
recalibrated to perform well when exploring the state space on the
fly. Remke et al.~\cite{RemkeHC07} have developed algorithms for model
checking CSL against infinite-state CTMCs of Quasi-birth-death processes (QBDs) and 
Jackson queuing networks (JQNs). The
systems to which the method is applicable are less general, but the
approach is less expensive than our method. To the best of the authors
knowledge, the method to bound the steady-state probabilities using
Lyapunov functions is used here for the first time for general MPMs.

\section{Conclusion}
\label{sec:conclusion}
In this paper, we have shown how to model check CSL on Markov
population models of infinite size. Without the steady-state operator,
the method is also applicable for general CTMC derived from a
high-level specification. We have evaluated our method on models from
the biological domain. The method extends previous
related publications by means to check the steady-state operator and
gives guarantees for the truth value obtained.

As future work we plan to integrate these methods into our probabilistic online model checker:
\iscasmc:
\begin{center}
\url{http://iscasmc.ios.ac.cn/IscasMC/}
\end{center}

\bibliographystyle{eptcs}
\bibliography{bib}

\end{document}